\documentclass{c0LSLag}

\begin{document}

\title{An Arnold-type principle for non-smooth objects}
\date{\today}
\author{Lev Buhovsky, Vincent Humili\`ere, Sobhan Seyfaddini}
\maketitle

\begin{center}
   \emph{Dedicated to Claude Viterbo}

      \emph{on the occasion of his 60th birthday.}
\end{center}

\bigskip
\begin{abstract} 
In this article we study the Arnold conjecture in settings where objects under consideration are no longer smooth but only continuous.  The example of a Hamiltonian homeomorphism, on any closed symplectic manifold of dimension greater than 2, having only one fixed point shows that the conjecture does not admit a direct generalization to continuous settings. However, it appears that the following  Arnold-type principle continues to hold in $C^0$ settings:  Suppose that $X$ is a non-smooth object for which one can define spectral invariants. If the number of spectral invariants associated to $X$ is smaller than the number predicted by the (homological) Arnold conjecture, then the set of fixed/intersection points of $X$ is homologically non-trivial, hence it is infinite. 

We recently proved that the above principle holds for Hamiltonian homeomorphisms of closed and aspherical symplectic manifolds.  In this article, we verify this principle in two new settings:  $C^0$ Lagrangians in cotangent bundles and Hausdorff limits of Legendrians in $1$-jet bundles which are isotopic to 0-section.

An unexpected consequence of the result on Legendrians is that the classical Arnold conjecture does hold for  Hausdorff limits of Legendrians in $1$-jet bundles.
\end{abstract}

\tableofcontents


\section{Introduction and main results}

The Arnold conjecture states that a Hamiltonian \emph{diffeomorphism} of a closed and connected symplectic manifold $(M, \omega)$ must have at least as many fixed points as the minimal number of critical points of a smooth function on $M$. The classical Lusternik-Schnirelmann theory shows that this minimal number is always at least the \emph{cup length} of $M$, a topological invariant of $M$ defined as\footnote{Here, $\cap$ refers to the intersection product in homology. The cup length can be equivalently defined in terms of the cup product in cohomology.}
\begin{align*}\cl(M):= \max \{k+1\,:\, \exists \, a_1&, \ldots, a_k \in H_*(M),\,\,
 \forall i, \deg(a_i)\neq  \mathrm{dim}(M)\\ &\text{ and }  a_1 \cap \cdots \cap a_k \neq 0\}.
\end{align*}

Therefore, a natural interpretation of the Arnold conjecture, sometimes referred to as the homological Arnold conjecture, is that a Hamiltonian diffeomorphism of $(M, \omega)$ must have at least $\cl(M)$ fixed points.\footnote{Note that we do not make any assumptions regarding non-degeneracy of Hamiltonian diffeomorphisms here.}  Successful efforts at resolving this conjecture were pioneered by Floer \cite{floer86, floer88, floer89} and led to the development of what is now called Floer homology.  The original version of the Arnold conjecture has been proven on symplectically aspherical manifolds \cite{rudyak-oprea}, \cite{floer89b}, \cite{hofer} while the homological version has been proven on a larger class of manifolds, {\it e.g.}\ $\mathbb{C}P^n$ by Fortune-Weinstein \cite{FW}, and symplectic manifolds which are negatively monotone by L\^e-Ono \cite{Le-Ono}.

The Arnold conjecture admits reformulations for symplectic objects other than Hamiltonian diffeomorphisms:  For example, a Lagrangian version of the conjecture states that in a cotangent bundle $T^*N$, a Lagrangian submanifold which is Hamiltonian isotopic to the zero section must have at least $\cl(N)$ intersection points with the zero section $O_N$ (See \cite{hofer,laudenbach-sikorav}). Here is a Legendrian reformulation of this last statement: a Legendrian submanifold in a 1-jet bundle $J^1N=T^*N\times \R$,  which is isotopic to the zero section through Legendrians, must have at least $\cl(N)$ intersections with the 0-wall $O_N\times\R$.\footnote{Sandon has recently presented a reformulation of the Arnold conjecture for contactomorphisms; see \cite{sandon12, sandon13}.}

The goal of this article is to understand the Arnold conjecture in settings where objects under consideration are no longer smooth but only continuous.  Although the Arnold conjecture is true for Hamiltonian homeomorphisms of surfaces \cite{matsumoto}, we showed in \cite{BHS} that every closed and connected symplectic manifold of dimension at least 4 admits a Hamiltonian homeomorphism with a single fixed point.  Analoguously, an example of a continuous Lagrangian submanifold Hamiltonian homeomorphic to the zero section and having a single intersection point with the zero section can be constructed in the cotangent bundle of any closed connected surface, see Proposition \ref{prop:single-intersection} below.

In spite of these counter-examples, 
it appears that certain reformulations of the Arnold conjecture do survive in $C^0$ settings.  These reformulations, which involve counting fixed/intersection points and certain ``homologically essential'' critical values of the action (\emph{i.e.\ spectral invariants}), are  inspired by the following statement from Lusternik--Shnirelman theory:
 
 \emph{
 Let $f$ be a smooth function on a closed manifold $M$.  If the number of homologically essential critical values of $f$ is smaller than $\cl(M)$, then the set of critical points of $f$ is homologically non-trivial. 
 }
 
  The above statement can be deduced from Proposition \ref{prop:cLS}. Homologically essential critical values, which are usually referred to as {\em spectral invariants} in the symplectic literature, are defined in Section \ref{sec:LS_theory}.   A subset $A\subset M$ is homologically non-trivial if for every open neighborhood $U$ of $A$ the map $i_*: H_j(U) \rightarrow H_j(M)$, induced by the inclusion $i: U \hookrightarrow M$, is non-trivial for some $j>0$. Clearly, homologically non-trivial sets are infinite.

 The reformulations of the Arnold conjecture which continue to hold in $C^0$ settings may be summarized as follows:

\begin{principle}\label{principle}
Suppose that $X$ is a non-smooth object for which one can define spectral invariants. If the number of spectral invariants associated to $X$ is smaller than the number predicted by the homological Arnold conjecture, then the set of fixed/intersection points of $X$ is homologically non-trivial, hence it is infinite. 
\end{principle}

In our recent article \cite{BHS2}, we established the above principle for Hamiltonian homeomorphisms of symplectically aspherical manifolds: Suppose that $(M, \omega)$ is closed, connected, and symplectically aspherical.  In Theorem 1.4 of \cite{BHS2} we prove that if $\phi$ is a Hamiltonian homeomorphism of $(M, \omega)$ with fewer spectral invariants than $\mathrm{cl}(M)$, then the set of fixed points of $\phi$ is homologically non-trivial. A variant of this statement for negative monotone symplectic manifolds and for complex projective spaces has been proven by Y. Kawamoto in \cite{kawamoto}.

The main results of this article establish Principle \ref{principle} in two more contexts: $C^0$ Lagrangians in cotangent bundles and Hausdorff limits of Legendrians in $1$-jet bundles. 

\medskip
\noindent \textbf{$C^0$ Lagrangians:}  Consider the cotangent bundle $T^*N$ of a closed manifold $N$ and denote by $O_N$ its zero section.  As we will see in Section \ref{sec:lagrangians},  (Lagrangian) spectral invariants can be defined for a $C^0$ Lagrangian of the form $L = \phi(O_N)$ where $\phi$ is a compactly supported Hamiltonian homeomorphism of $T^*N$; this is proven in Theorem \ref{prop:lag-spec-extension}.  We call such a $ C^0 $ Lagrangian ``a $C^0$ Lagrangian Hamiltonian homeomorphic to the zero section''. It is not difficult to see that in this setting our principle translates to the following statement.

\begin{theo} \label{theo:Arnold_Lagrangians}
Let $\phi$ denote a compactly supported Hamiltonian homeomorphism of $T^*N$ and suppose that $L = \phi(O_N)$.  If the number of spectral invariants of $L$ is smaller than $\cl(N)$, then $L \cap O_N$ is homologically non-trivial, hence it is infinite. 
\end{theo}

It is interesting to remark that, as for Hamitonian homeomorphisms, the Arnold conjecture breaks down for $C^0$ Lagrangians; this is the content of the next result.

  \begin{prop}\label{prop:single-intersection}
      Let $M$ be a closed connected surface. Then, there is a Hamitonian homeomorphism $\psi$ of $T^*M$ such that the $C^0$-Lagrangian $L=\psi(O_M)$ has only one intersection with the zero-section $O_M$.
  \end{prop}

Note that although we expect a similar statement to hold
 in higher dimensions, our proof is valid only for $M$ of dimension two. However, the argument we present is relatively simple compared to the construction in \cite{BHS}.

\begin{remark}  
  Of course, as a consequence of Theorem \ref{theo:Arnold_Lagrangians}, a $C^0$ submanifold $L$ as in Proposition \ref{prop:single-intersection} must have at least $\mathrm{cl}(N)$ distinct spectral invariants.
\end{remark}

\begin{remark}  It is reasonable to ask if in the above theorem the hypothesis $L = \phi(O_N)$ could be weakened to $L$ being the Hausdorff limit of a sequence $L_i$, where each $L_i$ is Hamiltonian isotopic to the zero section. This is related to a conjecture of Viterbo; see also Remark \ref{rem:Viterbo_conj} below.
\end{remark}
\medskip 

\noindent \textbf{Hausdorff limits of Legendrians:} As we will show in Section \ref{sec:legendrians}, one can associate spectral invariants to the Hausdorff limit of a sequence of Legendrians which are contact isotopic to the zero section in a $1$-jet bundle $J^1N$.  The interpretation of our principle in this case turns out to be particularly interesting for the following reason: Consider an intersection point $(q,0,z)$ between such a Legendrian $L$ and the 0-wall. This point corresponds to a critical point of the action and the associated critical value is $z$. In other words, the critical value can be read directly from the intersection point. It follows that in this context Principle \ref{principle} implies the Arnold conjecture itself!

As explained above, for a smooth Legendrian $L$ the action spectrum is given by $\spec(L)= \pi_{\mathbb{R}}(L \cap (O_N \times \mathbb{R}))$, where  $\pi_\mathbb{R} : J^1N = T^* N \times \mathbb{R} \rightarrow \mathbb{R} $ is the natural projection. By analogy, we will define the \emph{spectrum} of any subset $L\subset J^1N$ to be
\[\spec(L)= \pi_{\mathbb{R}}(L \cap (O_N \times \mathbb{R})).\]

\begin{theo} \label{theo:Arnold_Legendrians}
Let $L_i$ be a sequence of Legendrian submanifolds in $J^1N$ which are contact isotopic to the zero section $O_N \times\{0\} $. Suppose that this sequence has a limit $L$ for the Hausdorff distance, where $ L \subset J^1N $ is a compact subset.

 Assume that the cardinality $\spec(L)$ is strictly less than $ \cl(N) $. Then, there exists $\lambda\in\spec(L)$ such that $L\cap (O_N\times\{\lambda\})$ is homologically non-trivial in $O_N\times\{\lambda\}$. In particular, $L \cap (O_N \times \mathbb{R})$ is infinite. 
\end{theo}

Note that we make no assumptions with regards to regularity of $L$.  In fact, we do not even require $L$ to be a $C^0$ submanifold of $J^1N$.

\begin{remark}
  A careful examination of the proof of Theorem \ref{theo:Arnold_Legendrians} reveals that the assumption of Hausdorff convergence of $L_i$ to $L$ can be relaxed to the following: any neighborhood of $L$ contains $L_i$ for $i$ large.
\end{remark}

\begin{remark} In an ongoing project  \cite{Humiliere-Vichery}, the second author and N. Vichery show that Principle \ref{principle} can also be established for singular supports of sheaves (belonging to a certain subcategory of sheaves introduced by Tamarkin). These singular supports can be seen as (singular) generalizations of Legendrian submanifolds. 
\end{remark}

\subsection*{Organization of the paper}  
In Section \ref{sec:hamiltonian_homeos}, we recall some basic notions from symplectic geometry.   In Section \ref{sec:prelim_spec}, we introduce preliminaries on Lusternik-Schnirelmann theory and spectral invariants. 

Section \ref{sec:lagrangians} is dedicated to establishing Principle \ref{principle} for  $C^0$ Lagrangians Hamiltonian homeomorphic to the zero section.  The main technical step for doing so, which is of independent interest, consists of proving that Lagrangian spectral invariants can be defined for such $C^0$ Lagrangians.  This is achieved in Section \ref{sec:Spec-invar-C0-Lag}; see  Theorem \ref{prop:lag-spec-extension} therein.  Theorem \ref{theo:Arnold_Lagrangians} is proven in Section \ref{sec:proof_Lagrangians}. We prove Proposition \ref{prop:single-intersection} in Section \ref{sec:single-intersection}.
Lastly, Theorem \ref{theo:Arnold_Legendrians} is proven in Section \ref{sec:legendrians}.

\subsection*{Acknowledgments}
We dedicate this article to Claude Viterbo whose works in mathematics have deeply influenced ours.  Not only that, Claude's constant support and interest  in our research, since the very beginnings of our careers,  has been a great source of encouragement  to each one of us.   

Lemma \ref{lemma:variant_HLsword} was proven jointly with R\'emi Leclercq.  We are grateful to him for generously sharing his ideas with us.  We also thank Alberto Abbondandolo for pointing out to us the paper of Pierre Pageault \cite{Pageault} used in the proof of Proposition \ref{prop:single-intersection}. Our proofs of Theorems \ref{theo:Arnold_Lagrangians} and \ref{theo:Arnold_Legendrians} were inspired by the paper of Wyatt Howard \cite{Howard2012}. 

  The first author was partially supported by ERC Starting Grant 757585 and ISF Grant 2026/17. The second author was partially supported by the ANR project ``Microlocal'' ANR-15-CE40-0007.
  This material is based upon work supported by the National Science Foundation under Grant No. DMS-1440140 while the third author was in residence at the Mathematical Sciences Research Institute in Berkeley, California, during the Fall 2018 semester. The third author greatly benefited from the lively research atmosphere of the MSRI and would like to thank the members and staff of the MSRI for their warm hospitality.  The third author was partially supported by ERC Starting Grant 851701.

\section{Preliminaries from symplectic geometry}\label{sec:hamiltonian_homeos}
  For the remainder of this section $(M, \omega)$ will denote a connected symplectic manifold. Recall that a symplectic diffeomorphism is a diffeomorphism $\theta: M \to M$ such that $\theta^* \omega = \omega$.   The set of all symplectic diffeomorphisms of $M$ is denoted by $\Symp(M, \omega)$.  
  Hamiltonian diffeomorphisms constitute an important class of examples of symplectic diffeomorphisms.  These are defined as follows: A smooth Hamiltonian $H \in C_c^{\infty} ([0,1] \times M)$  gives rise to a time-dependent vector field $X_H$ which is defined via the equation: $\omega(X_H(t), \cdot) = -dH_t$.  The Hamiltonian flow of $H$, denoted by  $\phi^t_H$, is by definition the flow of $X_H$.  A  compactly supported Hamiltonian diffeomorphism is a diffeomorphism which arises as the time-one map of a Hamiltonian flow  generated by a compactly supported Hamiltonian.  The set of all compactly supported Hamiltonian diffeomorphisms is denoted by $\Ham_{c}(M, \omega)$; this forms a normal subgroup of $\Symp(M, \omega)$.
  
\subsection{Symplectic \& Hamiltonian homeomorphisms}
 We equip $M$ with a Riemannian distance $d$. Given two maps $\phi, \psi :M \to M,$ we denote
$$d_{C^0}(\phi,\psi)= \max_{x\in M}d(\phi(x),\psi(x)).$$
We will say that a sequence of compactly supported maps $\phi_i : M \rightarrow M$, $C^0$--converges to $\phi$, if there is a compact subset of $M$ which contains the supports of all $\phi_i$'s and if $d_{C^0}(\phi_i, \phi) \to 0$ as $ i \to \infty$. Of course, the notion of $C^0$--convergence does not depend on the choice of the Riemannian metric.

\begin{definition} \label{def:sympeo}
 A homeomorphism  $\theta : M \to M$  is said to be symplectic if it is the $C^0$--limit of a sequence of symplectic diffeomorphisms.  We will denote the set of all symplectic homeomorphisms by $\Sympeo(M, \omega)$.  
\end{definition}

The Eliashberg--Gromov theorem states that a symplectic homeomorphism which is smooth is itself a symplectic diffeomorphism. We remark that if $\theta$ is a symplectic homeomorphism, then so is $\theta^{-1}$.  In fact, it is easy to see that $\Sympeo(M, \omega)$ forms a group. 

\begin{definition}  \label{def:hameo} 
A symplectic homeomorphism $\phi $ is said to be a Hamiltonian homeomorphism if it is the $C^0$--limit of a sequence of Hamiltonian diffeomorphisms.  We will denote the set of all Hamiltonian homeomorphisms by $\overline{\Ham}(M, \omega)$. 
\end{definition}

  It is not difficult to see that  $\overline{\Ham}(M, \omega)$ forms a normal subgroup of  $\Sympeo(M, \omega)$.   It is a long standing open question whether a smooth Hamiltonian homeomorphism, which is isotopic to identity in $\Symp(M,\omega) $, is a Hamiltonian diffeomorphism;  this is often referred to as the $C^0$ Flux conjecture; see \cite{LMP, Sey13c, buhovsky14}.  
  
  We should add that alternative definitions for Hamiltonian homeomorphisms do exist within the literature of $C^0$ symplectic topology.  Most notable of these is a definition given by M\"uller and Oh in \cite{muller-oh}. A homeomorphism which is Hamiltonian in the sense of \cite{muller-oh} is necessarily Hamiltonian in the sense of Definition \ref{def:hameo} and thus, the results of this article apply to the homeomorphisms of \cite{muller-oh} as well. 
  
  \subsection{Hofer's distance}\label{sec:hofer_distance}  
  We will denote the Hofer norm on $C_c^{\infty}([0,1] \times M)$ by  \[ \| H \| = \int_0^1 \left( \max_{x \in M} H(t,\cdot) - \min_{x \in M} H(t, \cdot)\right) dt.\]   The Hofer distance on $\Ham(M, \omega)$ is defined via $$d_{\mathrm{Hofer}}(\phi, \psi)= \inf \Vert H-G\Vert,$$ where the infimum is taken over all $H, G$ such that $\phi^1_H = \phi$ and $\phi^1_G = \psi$.   This defines a bi-invariant distance on $\Ham(M, \omega)$.  
  
  Given $B\subset M$, we define its \emph{displacement energy} to be $$e(B):= \inf \{ d_{\mathrm{Hofer}}(\phi, \id): \phi \in \Ham(M, \omega), \phi(B) \cap B = \emptyset\}.$$
Non-degeneracy of the Hofer distance is a consequence of the fact that $e(B) >0$ when $B$ is an open set.  This fact was proven in \cite{hofer90, polterovich93, lalonde-mcduff}.

\section{Preliminaries on spectral invariants}\label{sec:prelim_spec}

 We fix a ground field $\F$, e.g.  $\Z_2, \mathbb{Q}$, or $\mathbb{C}$. Singular homology, Floer homology and all notions relying on these theories depend on the field $\F$. 

\subsection{Min-max critical values and Lusternik-Schnirelmann theory}\label{sec:LS_theory}

Let $f \in C^{\infty}(M)$ a smooth function on a closed and connected manifold $M$. For any $a \in \R$, let $M^{a} = \{x \in M: f(x) < a \}$.  Let 
$\alpha \in H_*(M)$ be a non-zero singular homology class and define 
$$\cLS(\alpha,f) := \inf \{a \in \R: \alpha \in \mathrm{Im} (i_a^*) \},  $$
where  $i_a^*: H_*(M^a) \rightarrow H_*(M)$ is the map induced in homology by the natural inclusion $i_a : M^a \hookrightarrow M$.
  The number $\cLS(\alpha,f)$  is a critical value of $f$ and such critical values are often referred to as \emph{homologically essential} critical values.

The function $\cLS  : H_*(M) \setminus \{0\} \times C^{\infty}(M) \rightarrow \R$ is called a \emph{min-max} critical value selector.  In the following proposition $[M]$ denotes the fundamental class of $M$ and $[pt]$ denotes the class of a point.

\begin{prop}\label{prop:cLS}
The min-max critical value selector $\cLS$ possesses the following properties.
\begin{enumerate}
\item $\cLS(\alpha, f)$ is a critical value of $f$,
\item $\cLS([pt],f) = \min(f) \leq \cLS(\alpha,f) \leq \cLS([M],f) = \max(f)$, 
\item $\cLS(\alpha\cap\beta, f) \leq \cLS(\alpha, f),$ for any $\beta \in H_*(M)$ such that $\alpha\cap\beta\neq 0$,
\item Suppose that $\deg(\beta)<\dim(M)$ and  $\cLS( \alpha \cap \beta, f) = \cLS (\alpha, f)$. Then, the set of  critical points of $f$ with critical value $\cLS (\alpha, f)$ is homologically non-trivial.
\end{enumerate}
\end{prop}

The above are well-known results from Lusternik-Schnirelmann theory and hence we will not present a proof here.  For further details, we refer the reader to \cite{LS, cornea-lupton-oprea, viterbo}.

\subsection{Spectral invariants for Lagrangians}\label{sec:spec_inv_lag}
 
 Let $N$ be a closed manifold.  The canonical symplectic structure on the cotangent bundle $T^*N$ is induced by the form $\omega_0 =  -d \lambda$  where $\lambda = p\, dq$.  
 We will denote by $\mathrm{Lag}$ the space of Lagrangian submanifolds of $T^*N$ which are Hamiltonian isotopic to the zero section, i.e.\  $\mathrm{Lag}:= \{\phi(O_N): \phi \in \Ham_c(T^*N, \omega_0)\}$.
 
  Consider $\phi \in \Ham_c(T^*N, \omega_0)$ and let $L = \phi(O_N)$.  We will briefly explain how one may associate Lagrangian spectral invariants to the Hamiltonian diffeomorphism $\phi$.  Pick a compactly supported Hamiltonian $H\in C^\infty_c([0,1]\times T^*N)$ such that $\phi = \phi^1_H$. The  action functional associated to $H$ is defined by
\begin{align*}
   \mathcal{A}_H : \Omega(T^*N) \rightarrow \R \;, \quad z \mapsto \int_0^1 H_t(z(t))\, dt - \int z^*\lambda
\end{align*}
where $\Omega(T^*N)=\{ z : [0,1] \rightarrow T^*N \,|\, z(0) \in O_N, \; z(1) \in O_N \}$. The critical points of $\mathcal{A}_H$ are the chords of the Hamiltonian vector field $X_H$ which start and end on $O_N$. Note that such chords are in one-to-one correspondence with $L \cap O_N$.  The spectrum of $\mathcal{A}_H$ consists of the critical values of $\mathcal{A}_H$. It is a nowhere dense subset of $\R$ which turns out to depend only on the time--1 map $\phi_H^1$, hence we will  denote it by $\Spec(L; \phi)$.

Now, using Lagrangian Floer homology, in a manner similar to what was done in the previous section, one can define a mapping $$ \ell : H_*(N) \setminus \{0\} \times \Ham_c(T^*N, \omega_0) \rightarrow \R$$ 
which associates to a homology class $a \in  H_*(N) \setminus \{0\}$ a value in $\Spec(L; \phi)$.  These numbers are often referred to as the Lagrangian spectral invariants of $\phi$.  They were first introduced by Viterbo in \cite{viterbo} via generating function techniques.  The Floer theoretic approach was carried out by Oh \cite{Oh99}.  Lagrangian spectral invariants have many properties some of which are listed below.  For a more comprehensive list of their properties, as well as a survey of their construction, we refer the reader to \cite{MVZ}; see for example Theorems 2.11 and  2.17 in \cite{MVZ}.

\begin{prop}\label{prop:Lag_spec}  The map $ \ell : H_*(N) \setminus \{0\} \times \Ham_c(T^*N, \omega_0) \rightarrow \R,$ satisfies the following properties:
 \begin{enumerate}
 
 \item  $\ell(a, \phi) \in \Spec(L; \phi)$,
 \item $\vert \ell(a, \phi^1_H) - \ell(a, \phi^1_G) \vert \leq \Vert H-G \Vert$,
 
 \item $ \ell(a \cap b, \phi \psi) \leq  \ell(a, \phi) + \ell(b, \psi)$, 
 
 \item $\ell([pt], \phi) \leq  \ell(a, \phi) \leq \ell([N], \phi)$, 
 
 \item $\ell([N], \phi) = - \ell([pt], \phi^{-1})$,
 \item If $\phi(O_N) = \psi(O_N)$, then $\exists \; C \in \R$ such that $\ell(a,\phi) = \ell(a, \psi) + C$ for all $a \in H_*(N) \setminus \{0\}$, 
 \item Suppose that $f: N \rightarrow \R$ is a smooth function and define the Lagrangian $L_f: = \{(q, \partial_qf(q)): q \in N \}$.   
   Denote by $F$ any compactly supported Hamiltonian of $T^*N$ which coincides with  $\pi^*f=f\circ \pi$ on a ball bundle $T^*_RN$ of $T^*N$ containing $L_f$.
   Then,  $\ell(a, \phi^1_F) = c_{LS}(a, f)$ for all $a \in H_*(N) \setminus \{0\}.$
 \item For any other manifold $N'$, the spectral invariants on $T^*(N\times N')$ satisfy
   \[\ell(a\otimes a',\phi\times\phi')=\ell(a,\phi)+\ell(a',\phi'),\] for all $\phi\in\Ham_c(T^*N,\omega)$, $\phi'\in\Ham_c(T^*N',\omega)$, $a\in  H_*(N) \setminus \{0\}$ and $a'\in  H_*(N') \setminus \{0\}$.
\end{enumerate}
 \end{prop}

 Note that the sixth property above tells us that spectral invariants $ \ell(a, \phi)$ are essentially invariants of the Lagrangian $L := \phi(O_N)$.   As a consequence of this property, the set of spectral invariants of $L $ is well-defined upto a shift by a constant.  In particular, we can make sense of the total number of spectral invariants of any Lagrangian $L$ which is Hamiltonian isotopic to the zero section. Similarly,  we see that $ \gamma: \mathrm{Lag} \rightarrow \R,$ defined by  
 
 \begin{equation}\label{eq:Lag_gamma}
 \gamma( \phi(O_N)):= \ell([N], \phi) -\ell([pt], \phi)
 \end{equation} is well-defined, \emph{i.e.}\ it only depends on the Lagrangian $\phi(O_N)$ and not on $\phi$. Viterbo showed in \cite{viterbo} that $\gamma$ induces a non-degenerate distance on $\mathrm{Lag}$.

 Finally, we should mention that Lagrangian spectral invariants have been constructed in settings more general than what is described above by Leclercq \cite{Lec} and Leclercq-Zapolsky \cite{Lec-Zap}.
 
 \bigskip
 
\noindent \textbf{Hamiltonian Spectral Invariants:}  In order to prove that Lagrangian spectral invariants can be defined for $C^0$ Lagrangians Hamiltonian homeomorphic to the zero section, that is to prove Theorem \ref{prop:lag-spec-extension} below, we will need to use certain results from the theory of Hamiltonian spectral invariants.  Here, we will briefly recall the aspects of this theory which will be needed below.   For further details on the construction of these invariants see \cite{schwarz, Oh05}.  The specific result used here, which compares Lagrangian and Hamiltonian spectral invariants, was proven in \cite{MVZ}.

Given $\phi \in \Ham_c(T^*N, \omega_0)$ and $a \in H_*(N) \setminus \{0\}$, using Hamiltonian Floer homology, one can define the Hamiltonian spectral invariant $c(a, \phi)$; this is a real number which belongs to the (Hamiltonian) action spectrum of $\phi$, {\it i.e.}\ there exists a fixed point of $\phi$ whose action is  the value $c(a, \phi)$.  These spectral invariants satisfy a list of properties similar to those listed in Proposition \ref{prop:Lag_spec}.  We will be needing the following property which is proven in \cite{MVZ}: For any $\phi \in \Ham_c(T^*N, \omega_0)$ and any $a \in H_*(N) \setminus \{0\}$ we have  
\begin{equation}\label{eq:MVZ_inequality}
c([pt], \phi) \leq \ell(a, \phi) \leq c([N], \phi).
\end{equation}
 See Proposition 2.14 and item  \emph{iv} of Theorem 2.17 in \cite{MVZ}. 

Similarly to Equation \eqref{eq:Lag_gamma}, we define $\gamma: \Ham_c(T^*N, \omega_0) \to \R$ via 
\begin{equation}\label{eq:Ham_gamma}
\gamma(\phi) : = c([N], \phi) - c([pt], \phi).
\end{equation}
Like its Lagrangian cousin, $\gamma$ induces a non-degenerate distance on $\Ham_c(T^*\\N, \omega_0)$.  We will need the following properties:

\begin{enumerate}
\item \textbf{Comparison Inequality:} As an immediate consequence of Equation \ref{eq:MVZ_inequality}, the Lagrangian version of $\gamma$ is smaller than the Hamiltonian version.  More precisely, for any $\phi \in \Ham_c(T^*N, \omega_0)$  we have
\begin{equation}\label{eq:MVZ_inequality_gamma}
 \gamma( \phi(O_N) ) \leq \gamma(\phi).
\end{equation}

\item  \textbf{Conjugacy Invariance:} For any $\phi \in \Ham_c(T^*N, \omega_0)$ and any symplectic diffeomorphism  $\psi$ of $T^*N$, we have
\begin{equation}\label{eq:conj_invariance_gamma}
\gamma( \phi) = \gamma( \psi \phi \psi^{-1}).
\end{equation}

\item \textbf{Triangle Inequality:}  For any $\phi, \psi \in \Ham_c(T^*N, \omega_0)$, we have 

\begin{equation}\label{eq:triangle}
\gamma(\phi \psi) \leq \gamma(\phi) + \gamma(\psi).
\end{equation}

\item \textbf{Energy-Capacity Inequality:}  Suppose that the support of $\phi$ can be displaced, then 
\begin{equation}\label{eq:energy_capacity}
\gamma(\phi) \leq 2 e(\mathrm{supp}(\phi)),
\end{equation}
where $e(\mathrm{supp}(\phi))$ is the displacement energy of $\mathrm{supp}(\phi)$.
\end{enumerate}

  \subsection{Spectral invariants for Legendrians via generating functions}\label{sec:spec_inv_leg}
Once again let $N$ be a closed manifold.  The standard contact structure on the 1-jet bundle $J^1N = T^*N \times \R$ is induced by the contact form $\alpha = dz - \lambda $, where $z$ is the coordinate on $\R$.  We will denote by $\mathrm{Leg}$ the space of Legendrian submanifolds of $J^1N$ which are contact isotopic to the zero section.   It was proven by Chaperon \cite{chaperon} and Chekanov \cite{chekanov} that  for every $L \in \mathrm{Leg}$ there exists a generating function quadratic at infinity (gfqi) $S : N \times E \rightarrow \R$, where $E$ is some auxiliary vector space, such that 
$$L = \left\{\left(q, \frac{\partial S}{\partial q}(q, e), S(q,e)\right)\,:\, \frac{\partial S}{\partial e} (q,e) = 0 \right\}.$$

Observe that critical points of $S$ correspond to the intersection points of $L$ with the zero wall $O_N \times \R$: $(q,e)$ is a critical point of $S$ if and only if $(q, 0, S(q,e))$ is a point on $L$.  Note that one can obtain the critical value of a given critical point of $S$ by simply reading the $z$--coordinate of the corresponding intersection point of $L$ with the zero wall.

By applying a min-max construction similar to that of Section \ref{sec:LS_theory} to the gfqi $S$, one can define Legendrian spectral invariants of the Legendrian $L$:
$$ \ell : H_*(N) \setminus \{0\} \times \mathrm{Leg} \rightarrow \R.$$
The fact that $\ell(a, L)$ does not depend on the choice of the gfqi $S$ is a consequence of the uniqueness theorem of Th\'eret and Viterbo \cite{theret, viterbo}. For further details on the construction see \cite{Zapolsky}.

We will now state those properties of Legendrian spectral invariants which will be used below.  

\begin{prop}\label{prop:Leg_spec}[See \cite{Zapolsky}]  The map $ \ell : H_*(N) \setminus \{0\} \times \mathrm{Leg} \rightarrow \R,$ satisfies the following properties:
 \begin{enumerate}
 \item  $\ell(a, L)$ is a critical value of the corresponding  gfqi $S$,

 \item The map $\ell(a,\cdot): \mathrm{Leg}\rightarrow \R$ is continuous with respect to the $C^{\infty}$ topology,
 \item $ \ell(a \cap b, L + L') \leq \ell(a,L) + \ell(b, L')$, for all $L, L'\in\mathrm{Leg}$ such that $L + L' : = \{(q, p+ p', z + z'): (q,p,z) \in L, (q, p', z') \in L' \}$ is a smooth Legendrian submanifold contact isotopic to the $0$-section. 
 \item Suppose that $f: N \rightarrow \R$ is a smooth function and define the Legendrian $L_f: = \{(q, \partial_qf(q), f(q)): q \in N \}$. Then,  $\ell(a, L_f) = c_{LS}(a, f)$ for all $a \in H_*(N) \setminus \{0\}$.
\end{enumerate}
 \end{prop}

\begin{remark}
A proof of item 3 in Proposition  \ref{prop:Leg_spec} is based on the following observation: If $S, S'$ are gfqi's for $L, L'$, respectively, then $S \oplus S' :N \times E \times E' \rightarrow \R$ defined  by $S\oplus S'(q,e,e'):=S(q,e)+S'(q,e')$ is a gfqi for the Legendrian  $L + L'$.
\end{remark}


 \section{$C^0$ Lagrangians, proof of Theorem \ref{theo:Arnold_Lagrangians}  and Proposition \ref{prop:single-intersection}}\label{sec:lagrangians}

 The first two subsections in this section are devoted to the proof of Theorem \ref{theo:Arnold_Lagrangians}. In the third, we prove Proposition \ref{prop:single-intersection}.
We begin by giving a precise definition of compactly supported Hamiltonian homeomorphisms of $T^*N$.

  Equip $N$ with a Riemannian metric and denote by $T^*_rN:= \{(q, p) \in T^*N: \|p \| < r\}$ the cotangent disc bundle of radius $r >0$.  We define $\Ham_c(T^*_rN, \omega_0)$ to be the set of Hamiltonian diffeomorphisms whose support is  contained in $T^*_rN$.  A compactly supported Hamiltonian homeomorphism is a homeomorphism which belongs to the uniform closure of $\Ham_c(T^*_rN, \omega_0)$ for some $r>0$; we will denote their collection by  $\overline{\Ham}_c(T^*N, \omega_0)$.

\subsection{Spectral invariants for $C^0$ Lagrangians} \label{sec:Spec-invar-C0-Lag}
We will now prove that Lagrangian spectral invariants can be defined for $C^0$ Lagrangians of the form $L= \phi(O_N)$ where $\phi \in \overline{\Ham}_c(T^*N, \omega_0)$.   Below is the continuity result which allows us to define spectral invariants for such $C^0$ Lagrangians. 

\begin{theo}\label{prop:lag-spec-extension}
   Lagrangian spectral invariants satisfy the following two properties:  
\begin{enumerate}
\item  For any homology class $a\in  H_*(N) \setminus \{0\}$, the map \[ \ell(a,\cdot) :\Ham_c(T^*N, \omega_0) \rightarrow \R\] is continuous with respect to the $C^0$ topology on $\Ham_c(T^*N, \omega_0)$ and extends continuously to the closure $\overline{\Ham}_c(T^*N, \omega_0)$.
\item  If $\phi(O_N) = \psi(O_N)$, then $\exists \; C \in \R$ such that $\ell(a,\phi) = \ell(a, \psi) + C$ for all $a \in H_*(N) \setminus \{0\}$ and for any $\phi, \psi \in \overline{\Ham}_c(T^*N, \omega_0)$.
\end{enumerate}
\end{theo}

Note that as a consequence of the second item, we can define the spectral invariants of a $C^0$ Lagrangian Hamiltonian homeomorphic to the zero section, upto shift.  In particular, it makes sense to speak of the number of spectral invariants of such a $C^0$ Lagrangian.

\medskip

The first part of the above theorem follows from techniques which have by now become rather standard in $C^0$ symplectic topology and hence, we will only sketch a proof of this part of the theorem. The second part of the statement, however, is based on a trick  which was recently introduced in our article \cite{BHS2} in the course of proving $C^0$ continuity of spectral invariants for Hamiltonian diffeomorphisms; see Theorem 1.1 therein.

\begin{proof}[Proof of Theorem \ref{prop:lag-spec-extension}]
  We begin with the proof of the first statement. We will be needing the following claim.

\begin{claim}\label{cl:epsilon_shift}  For every $r>0$,  there exist constants $C, \delta >0$,  depending on $r$, such that for any $\psi \in \Ham_c(T^*_rN, \omega_0)$, if $d_{C^0}(\id, \psi) \leq \delta$, then $|\ell(a, \psi)| \leq C d_{C^0}(\id, \psi)$. 
\end{claim}
\begin{proof}[Proof of Claim \ref{cl:epsilon_shift}]

As a consequence of Inequality \eqref{eq:MVZ_inequality}, it is sufficient to prove the result for the Hamiltonian spectral invariants. This is proved in \cite{Sey12}  in the case of symplectically aspherical closed manifolds; see see Theorem 1 therein. The proof given in \cite{Sey12} easily adapts to our settings.  
\end{proof}

Claim \ref{cl:epsilon_shift} proves continuity of our map at the identity.  Next, we consider $\id \neq \phi \in \Ham_c(T^*_rN, \omega_0)$.  We leave it to the reader to check that Properties 3, 4 and 5 in Proposition \ref{prop:Lag_spec} yield the following: $$| \ell(a, \phi \psi) - \ell(a, \phi) |\leq  \max\{ |\ell([N], \psi)|, |\ell([pt], \psi)|\}.$$  
Combining this with Claim \ref{cl:epsilon_shift} we conclude that for any $\phi, \psi \in \Ham_c(T^*_rN, \omega_0)$
$$d_{C^0}(\id, \psi) \leq \delta \implies |\ell(a, \phi \psi) - \ell(a, \phi)| \leq C d_{C^0}(\id, \psi).$$
This proves that $\ell(a, \cdot) :  \Ham_c(T^*_rN, \omega_0) \rightarrow \R$ is locally Lipschitz continuous.  Hence, it extends continuously to the closure $\overline{\Ham}_c(T^*_rN,\omega_0)$.  This finishes the proof of the first statement of the theorem.
\end{proof}

\medskip

We now turn our attention to the second statement of the theorem.  We begin with the following apriori weaker statement.

\begin{theo}\label{lem:zero_section} Let $\phi\in  \overline{\Ham}_c(T^*N, \omega_0)$ be a Hamiltonian homeomorphism. 
If $\phi(O_N) = O_N$, then there exists a constant $C$ such that $\ell(a, \phi) = C$  for all $a \in H_*(N) \setminus \{0\}$.
\end{theo}

 Note that in the case where $\phi$ is a smooth Hamiltonian diffeomorphism, the above theorem reduces to Property 6 in Proposition \ref{prop:Lag_spec}.

\begin{remark}\label{rem:Viterbo_conj} It can be checked that Theorem \ref{lem:zero_section} is a consequence of the following conjecture of Viterbo: If $L_i \subset T^*N$ is a sequence of Lagrangians Hamiltonian isotopic to the zero section, which Hausdorff converges to the zero section $O_N$, then $\gamma(L_i) \to 0$.  This conjecture has been established in several case by Shelukhin, e.g. $N=S^n, \C P^n, \T^n$ and others; See \cite{shelukhin-zoll,shelujhin-string}.
\end{remark}
Let us prove that the result follows from the above theorem.  Suppose that $\phi(O_N) = \psi(O_N)$.  First, note that, as a consequence of the third item in Proposition \ref{prop:Lag_spec}, we have the following inequality: 
$$-\ell([N],  \phi^{-1} \psi ) \leq \ell(a, \phi) -\ell(a, \psi) \leq \ell([N],  \psi^{-1} \phi).$$
Hence, it is sufficient to show that $\ell([N],  \psi^{-1} \phi)= -\ell([N], \phi^{-1} \psi )$.  Now, by the fifth item of Proposition \ref{prop:Lag_spec}, $-\ell([N], \phi^{-1} \psi ) = \ell([pt],  \psi^{-1} \phi)$ and by Theorem \ref{lem:zero_section} we have $\ell([pt],\psi^{-1} \phi )= \ell([N], \psi^{-1} \phi )$. 

It remains to prove Theorem \ref{lem:zero_section}.   The proof we present below relies on  an idea similar to what was used in the proof of Theorem 1.1 of \cite{BHS2}.

\begin{proof}[Proof of Theorem \ref{lem:zero_section}]  Pick a sequence  $\phi_i$ in $ \Ham_c(T_\rho^*N, \omega_0)$ which converges uniformly to $\phi$ (for some $ \rho > 0 $).    By Theorem \ref{prop:lag-spec-extension}, it is enough to show that there exists a constant $C$ such that $ \ell(a, \phi_i) \to C$ for any $a\in H_*(N) \setminus \{0\}$. Denote $L_i := \phi_i(O_N)$ and observe that, as a consequence of the fourth property in Proposition \ref{prop:Lag_spec}, it is sufficient to show that $\gamma(L_i)$  converges to zero.

  As we will now explain, we may assume without loss of generality that $\phi$ admits a fixed point on the zero section $O_N$. Indeed, fix $p\in O_N$ and pick a Hamiltonian $G$ which vanishes on the zero section such that $\phi\circ \phi_G^1(p)=p$. For all $i$, we have $\gamma(\phi_i\circ\phi_G^1)=\gamma(\phi_i)$, by the sixth item of Proposition \ref{prop:Lag_spec}. Thus, we can replace $\phi_i$ by $\phi_i\circ\phi_G^1$ and $\phi$ by $\phi\circ\phi_G^1$.

Observe that the Lagrangians $L_i$ converge in Hausdorff topology to the zero section, \emph{i.e.}\ for any $\delta > 0$ we have $L_i \subset T^*_{\delta} N$ for $i$ sufficiently large.    We will reduce the theorem to the following lemma which was obtained jointly with R. Leclercq.  A variant of this lemma was established in \cite{HLS13}; see Lemma 8 therein.
  
  Given $B \subset N$, we denote $T^*B := \{(q,p) \in T^*N: q \in B\}$ and  $O_B := \{(q,0) : q \in B\}$. 
\begin{lemma}\label{lemma:variant_HLsword} 
Let $L_i$ denote a sequence of Lagrangians in $T^*N$ which are Hamiltonian isotopic to $O_N$. Suppose that there exists a ball $B \subset N$ such that $L_i \cap T^*B = O_B$.  If the sequence $L_i$ Hausdorff converges to $O_N$, then $\gamma(L_i) \to 0$.
\end{lemma}
\begin{proof} Pick $\phi_i \in \Ham_c(T^*N, \omega_0)$ such that $\phi_i(O_N) = L_i$.  We begin with the following observation:    Since $L_i \cap T^*B$ is connected, any  two points $(q_1, 0), (q_2, 0) \in L_i \cap T^*B$ have the same action.  Let $C_i$ denote this value.

For any given $\eps > 0$,  pick a smooth function $f: N \to \R$ whose critical points are all contained in $B$ and such that $ \max (f) - \min(f) < \eps$.  Denote by $ \pi : T^*N \rightarrow N $ the natural projection and define $F = \beta \,\pi^*f$ where $\beta : T^*N \rightarrow [0,1]$ is  compactly supported and $\beta=1$ on $T_R^*N$ where $R \gg 1$.  

Note that  $\phi^t_F(q,p) = (q, p+ t \, df(q))$ for $t\in [0,1]$ and $(q, p) \in T^*_1N$.  Therefore, $\phi^1_F\phi_i (O_N) = L_i + L_{f}$ where $L_i + L_{ f} := \{(q, p +  df(q)): (q, p) \in L_i\}$.  The Hausdorff convergence of the sequence $L_i$ to $O_N$ and the fact that $L_i \cap T^*B =  O_B$ combine together to imply that $(L_i + L_f) \cap O_N = \{(q, 0): df(q) = 0\}$ for $i$ large enough.  

It is easy to see that the action of $(q, 0) \in (L_i + L_f) \cap O_N$ is given by $C_i + f(q)$ where $C_i$ is the constant introduced above.  Therefore, $$\gamma(L_i + L_f) \leq \max (f) -\min (f) <   \eps.$$
On the other hand,  by the second property from Proposition \ref{prop:Lag_spec}, we have $\vert \gamma(L_i + L_f) - \gamma(L_i) \vert \leq 2 (\max (f) -\min (f) ) < 2 \eps$.  Combining this with the previous inequality we obtain $\gamma(L_i) < 3 \eps$ for $i$ large enough which proves the lemma.
\end{proof}

\medskip
The end of the proof of Theorem \ref{lem:zero_section} will consist in reducing to Lemma \ref{lemma:variant_HLsword}. We will assume from now on that $N$ has even dimension. The case where $N$ has odd dimension reduces to the even dimensional case by replacing $N$ with $N\times\S^1$ and all $\phi_i$'s by $\phi_i\times\id_{\S^1}$. 

We introduce for that the auxiliary maps
\begin{align*}\Phi_i=\phi_i \times\phi_i^{-1}:\ T^*N \times T^*N &\to T^*N \times T^*N,\\
   (x,y)&\mapsto (\phi_i(x),\phi_i^{-1}(y)),
\end{align*}
where we endow $T^*N\times T^*N$ with the symplectic form $\omega_0 \oplus \omega_0$; observe that this is canonically symplectomorphic to $T^*(N \times N)$ equipped with its canonical symplectic structure.

Denote $\overline{L}_i := \phi_i^{-1}(O_N)$ and note that $\Phi_i(O_{N\times N}) =  L_i \times \overline{L}_i$.   The map $\Phi_i$ is a Hamiltonian diffeomorphism which is not compactly supported.  To obtain a compactly supported Hamiltonian diffeomorphism, we  cut off the generating Hamiltonian of $\Phi_i$ far away from $O_{N\times N}$ and obtain a new Hamiltonian diffeomorphism which we will continue to denote by $\Phi_i$.   It is not difficult to see  that $\Phi_i$ remains unchanged on a large enough neighborhood of the zero section and so $\Phi_i (O_{N \times N})$ continues to be $L_i \times \overline{L}_i$.

Properties 8 and 5 of Proposition \ref{prop:Lag_spec} yield
\begin{equation}\label{eq:Phi-phi-Lag}\gamma(L_i \times \overline{L}_i )=\gamma(L_i) +\gamma(\overline{L}_i) = 2 \gamma(L_i).
\end{equation}

   Our proof crucially relies on the following lemma. 

\begin{lemma}\label{lemma:trick_Lag} Fix  $\eps>0$.   We can find a ball $B \subset N$, and $\Psi_i \in \Ham_c(T^*N \times T^*N, \omega_0 \oplus \omega_0)$ such that the following properties hold :

  \begin{enumerate}[(i)]
  \item $\gamma(\Psi_i(O_{N\times N}))<\eps$ for $i$ sufficiently large,
  \item $\Psi_i \Phi_i(O_{N\times N})$ converges in Hausdorff topology to $O_{N\times N}$,
  \item $\Psi_i \Phi_i (O_{N\times N}) \cap T^* (B \times B) =  O_{B \times B}$ for $i$ sufficiently large. 
  \end{enumerate}
\end{lemma}

We now explain why this lemma implies that $\gamma(L_i) \to 0$.  Fix  $\eps>0$ and let $B$ and $\Psi_i$ be as provided by Lemma \ref{lemma:trick_Lag}.  Using (\ref{eq:Phi-phi-Lag}), the triangle inequality and the fifth property in Proposition \ref{prop:Lag_spec}, we get 
\begin{align*}
  \gamma(L_i) &= \tfrac12\gamma(L_i \times \overline{L}_i) =  \tfrac12\gamma(\Phi_i (O_{N\times N}))\\
& \leq \tfrac12\gamma(\Phi_i\circ\Psi_i(O_{N\times N}))+\tfrac12\gamma(\Psi_i^{-1}(O_{N\times N}))\\
  & < \tfrac12\gamma(\Phi_i\circ\Psi_i(O_{N\times N}))+ \tfrac\eps2.
\end{align*}

The second and the third items of Lemma \ref{lemma:trick_Lag} allow us to apply Lemma \ref{lemma:variant_HLsword} and conclude that $\gamma(\Phi_i\circ\Psi_i(O_{N\times N})) \to 0$.  This implies that $\gamma(L_i) \to 0$. This concludes the proof of Theorem \ref{lem:zero_section} assuming Lemma \ref{lemma:trick_Lag}.
\end{proof}

\medskip

\begin{proof}[Proof of Lemma \ref{lemma:trick_Lag}.]
Fix $\eps>0$.  Pick a non-empty open ball $B_1$  in $N\simeq O_N$ containing a fixed point $p$ of $\phi$ and such that the displacement energy  of $U_1:=T^*_{1} B_1$ in $T^*N$ is less than $\frac\eps{4}$.  Note that the displacement energy of $U_1 \times U_1$ inside $ T^*(N \times N) $ is also less than $\frac \eps{4}$.

The following claim asserts the existence of a convenient Hamiltonian diffeomorphism which switches coordinates on a small open set.  

\begin{claim}\label{claim:switch-coord-Lag} There exist an open ball $B_2\subset B_1$ containing the fixed point $p$, $0<r_2 < 1$ and a Hamiltonian diffeomorphism $f$ of $T^*N \times T^*N$ such that:
  \begin{itemize}
  \item $f(O_{N \times N}) = O_{N\times N}$,
  \item $f$ is the time-1 map of a Hamiltonian supported in $U_1 \times U_1$,
  \item for all $(x,y)\in U_2 \times U_2$, we have $f(x,y)=(y,x)$, where $U_2 :=T^*_{r_2} B_2$.
  \end{itemize}
\end{claim}
\begin{proof} Since $N$ is assumed even dimensional, there is an identity isotopy, say $\varphi_t$, of $N \times N$ which is supported in $B_1 \times B_1$ with the following property:  there exists a ball $B_2 \subset B_1$ containing $p$ such that $\varphi_1(q_1, q_2) = (q_2, q_1)$  on $B_2 \times B_2$. 

Let $\tilde \varphi_t$ denote the canonical  lift of this isotopy to $T^*N\times T^*N$.   The isotopy $\tilde \varphi_t$ is symplectic, it preserves $O_{N\times N}$, it is supported in $T^*B_1 \times T^*B_1$, and it can be checked that $\tilde \varphi_1(x, y) = (y, x)$  on $T^*B_2 \times T^*B_2$. Furthermore, the isotopy is Hamiltonian.  Let $H$ denote a generating Hamiltonian of the isotopy which is supported in $T^*B_1 \times T^*B_1$.

To construct our desired Hamiltonian diffeomorphism $f$, we simply replace $H$ by $\beta H$ where $\beta$ is a smooth cut-off function on $T^*(N \times N)$ such that $\beta = 1$ on $T^*_{1-\delta} (N \times N)$, where $\delta$ is a small positive number, and $\beta = 0$  outside $T^*_{1} (N \times N)$.  We set $f$ to be the time-$1$ map of the Hamiltonian flow of $\beta H$ and leave it to the reader to check that it satisfies the requirements of the claim.
\end{proof}

We can now complete the proof of Lemma \ref{lemma:trick_Lag}. Since $p\in B_2$, there exists a ball $B_3 \subset B_2 $ and $0 < r_3 < r_2$ such that $ \phi (U_3 )\Subset U_2$ (i.e., $ \phi(U_3) $ is compactly contained in $ U_2 $), where $U_3 : = T^*_{r_3} B_3$.

Let $\Upsilon_i=\phi_i \times \id_{T^*N}$ and let \[\Psi_i=\Upsilon^{-1}_i\circ f^{-1}\circ \Upsilon_i\circ f.\] 
We will first show that $\gamma(\Psi_i(O_{N \times N})) < \eps$.  Note that by Equation \eqref{eq:MVZ_inequality}, we have $\gamma(\Psi_i(O_{N \times N})) \leq \gamma (\Psi_i)$, where $\gamma(\Psi_i)$ is the Hamiltonian $\gamma$ which was introduced above in Equation \eqref{eq:Ham_gamma}.  Hence, it is sufficient to show that $\gamma(\Psi_i) < \eps$.
The triangle inequality for $\gamma$ (Equation \eqref{eq:triangle}) and its conjugacy invariance (Equation \eqref{eq:conj_invariance_gamma}) yield $\gamma(\Psi_i)\leq 2\gamma(f)$.  Lastly, $\gamma(f) < \frac \eps{2}$ because the displacement energy of its support is smaller than $\frac{\eps}{4}$; see Equation \eqref{eq:energy_capacity}. This implies Property (i) in Lemma \ref{lemma:trick_Lag}.

Next, we will verify the second property in Lemma \ref{lemma:trick_Lag}.  Define $\Psi :=\Upsilon^{-1} \circ f^{-1}\circ \Upsilon \circ f$, where $\Upsilon := \phi \times \id_{T^*N}$, and let $\Phi := \phi \times \phi^{-1}$.   Since  $f, \Upsilon$ and $\Phi$ preserve $O_{N  \times N}$, we conclude that $\Phi \circ\Psi$ also preserves $O_{N \times N}$.    Now, there exists  a neighborhood of $O_{N\times N}$ where the sequences $\Psi_i$ and $\Phi_i$  converge uniformly to  $\Psi $ and $\Phi$, respectively.  It follows that $\Phi_i \circ\Psi_i(O_{N \times N})$ converges in Hausdorff topology to $O_{N \times N}$.

It remains to verify the third property from the lemma.  We leave it to the reader to check that $\Phi_i\circ \Psi_i (x,y) =  (x, y)$ for all $(x,y) \in U_3  \times U_3$, when $ i $ is large enough.  This relies crucially on the following observations:  $f(x,y) = (y,x)$ on $U_2 \times U_2$ and $\Upsilon_i(U_3 \times U_3) \subset U_2 \times U_2$ for $i$ large enough.  The last statement is a consequence of the fact that $\phi(U_3) \Subset U_2$.
 
 Let $B=B_3 \times B_3$  and $r=r_3$, so that $T^*_r B = U_3 \times U_3$.  As we have seen,  for $i$ large, $\Phi_i \circ\Psi_i$ coincides with the identity on  $T^*_r B$.   We claim that this implies the third property.   Indeed, it clearly implies $O_B \subset \Phi_i \circ\Psi_i(O_{N \times N}) \cap T^*B$.  Furthermore, it also implies that if $\Phi_i \circ\Psi_i(O_{N \times N}) \cap T^*B$ contains a point which is not in $O_B$, then such a point is in $T^*B \setminus T_r^*B$.  But of course this cannot happen for $i$ large because of the Hausdorff convergence of  $\Phi_i \circ\Psi_i(O_{N \times N})$ to $O_{N \times N}$.    This establishes the third property in Lemma \ref{lemma:trick_Lag}.
\end{proof}

\subsection{Proof of Theorem \ref{theo:Arnold_Lagrangians}}\label{sec:proof_Lagrangians}

By the assumptions of the theorem, one can find some $ r > 0 $ and a sequence $ \phi_i \in \Ham_c(T^*_rN, \omega_0) $ such that $ \phi_i $ converges uniformly to $ \phi $. Since the number of Lagrangian spectral invariants of $ \phi $ is assumed to be less than $ \cl(N) $, there exist some $ \alpha,\beta \in H_*(N) $  with $ \deg \alpha, \deg \beta < \dim N $ and $ \alpha \cap \beta \neq 0 $, such that $ \ell(\alpha,\phi) = \ell(\alpha \cap \beta,\phi) =: \lambda $. By the continuity of spectral invariants ({\it i.e.}\ the first item of Theorem \ref{prop:lag-spec-extension}), we have $ \lim \ell(\alpha,\phi_i) = \lim \ell(\alpha \cap \beta,\phi_i) = \lambda $, when $ i \rightarrow \infty $.

Let $ U \subset O_N $ be any neighbourhood of $ L \cap O_N $ in $ O_N $. It is enough to show that the closure $ \overline{U} $ is homologically non-trivial in $ O_N $. For doing this, 
pick a smooth function $f: N \rightarrow \R$ such that $f = 0$ on $\overline{U}$ and $f < 0 $ on $N \setminus \overline{U}$. Denote by $ \pi : T^*N \rightarrow N $ the natural projection and define $F = \beta \pi^*f$ where $\beta : T^*N \rightarrow \R$ is  compactly supported and $\beta=1$ on $T_R^*N$ where $R$  is taken to be large in comparison to $r$.  

\begin{claim}\label{cl:claim3} 
There exists an integer $i_0$ such that for any $i\geq i_0$, and for sufficiently small values of $\eps > 0$, $$\ell(\alpha \cap \beta, \phi^{\eps}_F \, \phi_i ) = \ell(\alpha \cap \beta, \phi_i).$$ 
\end{claim}
\begin{proof} 
 Let $L_i = \phi_i(O_N)$ and $L_{\eps f} = \phi^{\eps}_F(O_N)$.  Note that  $\phi^t_F(q,p) = (q, p+ t \, df(q))$
for $t\in [0,1]$ and $(q, p) \in T^*_rN$.  Therefore, we have  $L_{\eps f} = \{(q, \eps df(q)) : q \in N\}$ and  $\phi^{\eps}_F \phi_i(O_N) = \phi^{\eps}_F(L_{i})= L_i + L_{\eps f}$ where $L_i + L_{\eps f} := \{(q, p + \eps df(q)): (q, p) \in L_i\}$. 
 
 Since $ L \cap \pi^{-1}(O_N \setminus U) $ is compact and does not intersect $ O_N $, and since the sequence $ \phi_i $ converges uniformly to $ \phi $, we conclude that for small enough $ \eps $ and large enough $ i $, $ (L_i +  L_{\eps f}) \cap \pi^{-1}(O_N \setminus U) $ does not intersect $ O_N $ as well. On the other hand, since $ f = 0 $ on $ U $, we get that $ (L_i +  L_{\eps f}) \cap \pi^{-1}(U) = L_i \cap \pi^{-1}(U) $. Therefore, for small enough $ \eps > 0 $ and large enough $ i $, the Lagrangians $ L_i $ and $ L_i + L_{\eps f} $ have the same intersection points with the zero section $ O_N $. Moreover, it is easy to see that for each such intersection point, the two action values corresponding to $ \phi_i $ and $ \phi^{\eps}_F \phi_i$ coincide. Therefore, by fixing $ i $ and $ \eps > 0 $, and considering the family of Lagrangians $ L_i + L_{s\eps f} $ when $ s \in [0,1] $, we see that the action spectra  $ \Spec(L_i + L_{s\epsilon f}, \phi^{s\eps}_F \phi_i) $ do not depend on $ s $. Also, recall that the action spectrum  has an empty interior in $ \mathbb{R} $. As a result, since the value $ \ell(\alpha \cap \beta, \phi^\eps_F \phi_i) $ depends continuously on $ s $, we conclude that it in fact does not depend on $ s \in [0,1] $. In particular, $ \ell(\alpha \cap \beta, \phi_i ) = \ell(\alpha \cap \beta, \phi^\eps_F \phi_i) $.
\end{proof}

The triangle inequality of Proposition \ref{prop:Lag_spec}  implies that, for all $i$,  $\ell(\alpha \cap \beta,  \phi^\eps_F \phi_i)-  \ell(\alpha , \phi_i) \leq \ell(\beta,  \phi^\eps_F )$.  Using the above claim, for $i$ large and $\eps$ small enough, we have $\ell(\alpha \cap \beta,\phi_i )-  \ell(\alpha , \phi_i) \leq \ell(\beta, \phi^\eps_F )$. Taking limit as $i \to \infty$, and recalling that  $ \lim \ell(\alpha,\phi_i) = \lim \ell(\alpha \cap \beta,\phi_i) = \lambda $, we  obtain $0\leq \ell(\beta,  \phi^\eps_F)$. 

We can now conclude our proof as follows. On the one hand, by Proposition \ref{prop:Lag_spec}.7,  we have $\ell(\beta,\phi_{\eps F})=c_{LS}(\beta, \eps f)=c_{LS}([M]\cap\beta, \eps f)$. On the other hand, Propostion \ref{prop:cLS}.2 gives $c_{LS}([M],\eps f)=0$. Thus, using Proposition \ref{prop:cLS}.3, we obtain  the equality $c_{LS}([M]\cap\beta, \eps f)=c_{LS}([M],\eps f)$. By Proposition \ref{prop:cLS}.4 it follows that the zero level set of $f$, that is $\overline{U}$, is homologically non-trivial.

\subsection{Proof of Proposition \ref{prop:single-intersection}}\label{sec:single-intersection}

 According to \cite{Pageault}, there exists a $C^1$ function $f:M\to\R$, whose set of critical points is an arc $\gamma$, i.e., is homeomorphic to $[0,1]$. Let us fix such a function $f$. Let $F=f\circ\pi$ where $\pi$ denotes the canonical projection $\pi:T^*M\to M$. The intersection between the $C^0$-Lagrangian submanifold $\mathrm{graph}(df)=\phi_F^1(O_M)$ and the zero-section is exactly $\gamma$ (where we canonically identify $O_M$ with $M$).

  We will construct the $C^0$-Lagrangian $L$ roughly by ``contracting the arc to a point''. More precisely, given a point $p\in \gamma$, we will construct a map $h:T^*M\to T^*M$ which is a symplectic diffeomorphism between $T^*M\setminus\gamma$ and $T^*M\setminus\{p\}$, and satisfies $h(\gamma)=p$ and $h(O_M)=O_M$. We will then prove that the map
  \[\psi:
    \begin{cases}
      x\mapsto h \phi_F^1 h^{-1}(x),& \text{for }x\neq p\\
      p\mapsto p. &  
    \end{cases}
  \]
  is a Hamiltonian homeomorphism and that $L=\psi(O_M)$ has a unique intersection point with $O_M$.
  
  Let us now start the construction. A version of the Jordan-Schoenflies theorem (for instance its extension due to Homma \cite{Homma}) implies that the arc $\gamma$ admits a basis of neighborhoods $(V_i)_{i\geq 0}$, which are all homeomorphic to open discs and satisfy $\overline{V_{i+1}}\subset V_i$ for all $i$. Let $(U_i)_{i\geq 1}$ be a decreasing basis of neighborhoods of $p$.
 Finally, let  $(\delta_i)_{i\geq 0}$ be a decreasing sequence  of real numbers converging  to 0.

 Let $W_0=V_{0}$ and $\eps_0=\delta_0$.
 Since the $V_i$'s form a basis of disc-like neighborhoods, there exists 
 a smooth (time-dependent) vector field $X_1$ supported in $W_0$ whose time-one map $\zeta_1$ sends 
 $V_1$ into $U_1$. We may also assume that $\zeta_1$ fixes $p$. We denote 
 $W_1=\zeta_1(V_{1}) \subset U_1$.

The Hamiltonian function $(q,p)\mapsto\langle p,X_1(q)\rangle$ vanishes on $O_M$ and its flow is supported in $T^*W_0$. By multiplying it with an appropriate cutoff function which equals 1 on a neighborhood of the support of $X_1$ in $T^*M$, we obtain a Hamiltonian $H_1$ supported in $T^*_{\eps_0}W_0$. This Hamiltonian $H_1$ vanishes on $O_M$, thus its flow preserves it. Moreover, by construction, the restriction of its flow to the zero section coincides with the flow of $X_1$. We denote by $h_1=\phi_{H_1}^1$ its time-one map. 

Repeating the above, we construct by induction 
a sequence of positive real numbers $\eps_k$ converging to $ 0 $,
a decreasing sequence of open subsets $(W_k)$ of $M$  and a sequence of Hamiltonians $(H_k)$ on $T^*M$ such that for each $k\geq 1$, the three following properties hold:
    \begin{enumerate}[(i)]
    \item $H_k$ is supported in $T_{\eps_{k-1}}^*W_{k-1}$,
    \item the time-one map $h_k=\phi_{H_k}^1$ preserves $O_M$,
    \item 
      $W_{k}=h_{k}\circ \dots \circ h_1(V_{k})$ is included in 
      $U_k$,
    \item $T_{\eps_{k}}^*W_{k}$ is included in 
      $h_{k}\circ \dots \circ h_1(T_{\delta_k}^*V_{k})$. 
    \end{enumerate}
    Indeed, assuming all the sequences built up to the order $k$, we let $X_{k+1}$ be a vector field on $M$ which maps  the disc  $h_{k}\circ \dots \circ h_1(V_{k+1})$ into 
    $U_{k+1}$. 
    The Hamiltonian $H_{k+1}$ is then obtained by cutting off  $(q,p)\mapsto \langle p,X_{k+1}(q)\rangle$ appropriately, as above.

    For any $x\in \gamma$, we have $h_k\circ \dots\circ h_1(x)\subset 
    U_{k}$
    thus the sequence $(h_k\circ \dots\circ h_1(x))$ converges to $p$. For any $x\notin \gamma$ we have $x\notin T_{\delta_k}^*
    V_k$ for $k$ large enough. It follows that for $k$ large enough, $h_k\circ \dots\circ h_1(x)$ does not belong to $T_{\eps_{k}}^*W_{k}$, hence does not belong to the support of any $h_i$ for $i>k$. Thus, the  sequence $(h_k\circ \dots\circ h_1(x))$ stabilizes to a point different from $p$.

    We set $h(x)=\lim_{k\to \infty}f_k(x)$, where $f_k(x):=h_k\circ \dots\circ h_1(x)$. This limit is uniform. Indeed, given $\eps>0$, there exists an integer $N$ such that $\mathrm{diam}(T_{\eps_k}^*
    U_k)<\eps$  for all $k\geq N$. Let $k\geq N$. Then for any $x\in f_k^{-1}(T_{\eps_k}^*W_k)$, we have $f_k(x)\in T_{\eps_k}^*W_k$, hence $f_{k+p}(x)\in T_{\eps_k}^*W_k\subset T_{\eps_k}^*
    U_k$ for any $p\geq 1$. Taking limit as $p$ goes to infinity, we obtain $f_k(x), h(x)\in T_{\eps_k}^*
    U_k$ hence $d(f_k(x), h(x))<\eps$. Now for $x\notin f_k^{-1}(T_{\eps_k}^*W_k)$ we have $f_k(x)\notin T_{\eps_k}^*W_k$, hence $f_{k+p}(x)=f_k(x)$ for all $p\geq 1$. We deduce that $h(x)=f_k(x)$. We have shown that for all $x$, $d(f_k(x),h(x))<\eps$, which proves that the limit is uniform.  

As a consequence, $h$ is continuous. Moreover the restriction of $h$ induces a  symplectic diffeomorphism $T^*M\setminus\gamma\to T^*M\setminus\{p\}$. Also note that $h$ preserves the zero section $O_M$. As announced in the beginning of the proof we now define
\[\psi:
    \begin{cases}
      x\mapsto h \phi_F^1 h^{-1}(x),& \text{for }x\neq p,\\
      p\mapsto p. &  
    \end{cases}
  \]
  Since $\phi_F^1(O_M)\cap O_M=\gamma$, and since $h(O_M)=O_M$, we have $\psi(O_M)\cap O_M=\{p\}$. Finally, $\psi$ is a Hamiltonian homeomorphism because it is the  $C^0$-limit of the Hamiltonian diffeomorphisms
  \[(h_k\circ \dots\circ h_1)\circ\phi_F^1\circ (h_k\circ \dots\circ h_1)^{-1}\]
as $k$ goes to infinity.  \hfill $\square$

\section{Hausdorff limits of Legendrians and proof of Theorem \ref{theo:Arnold_Legendrians}}\label{sec:legendrians}
This section is dedicated to the proof of Theorem \ref{theo:Arnold_Legendrians}.  Recall that we consider a sequence $L_i$ of Legendrian submanifolds, contact isotopic to the zero section in $J^1N = T^*N \times \R$, which has a Hausdorff limit $L$.  Denote by $ \pi_\mathbb{R} : J^1N = T^* N \times \mathbb{R} \rightarrow \mathbb{R} $ the natural projection.

We have not been able to verify whether it is possible to define Legendrian spectral invariants for the Hausdorff limit $L$.  However, as we will now explain, it is still possible to view Theorem \ref{theo:Arnold_Legendrians} as an incarnation of Principle \ref{principle}: Let $K$ be a (smooth) Legendrian submanifold of $J^1N$ which is contact isotopic to the zero section.  Then, as was explained in Section \ref{sec:spec_inv_leg}, the set $\spec(K)=\pi_{\R}(K \cap (O_N \times \R))$ is the set of critical values of the gfqi associated to $K$.  Hence, if the cardinality of $\spec(K)$ is smaller than $\cl(N)$, then so is the total number of spectral invariants of $K$.  Therefore, despite the fact that we cannot define spectral invariants for the Hausdorff limit $L$, we can interpret \emph{the cardinality of the set $\spec(L)=\pi_\R (L \cap (O_N \times \R)) $ being smaller than $\cl(N)$} to mean that \emph{$L$ has fewer spectral invariants than $\cl(N)$.}

\begin{proof}[Proof of Theorem \ref{theo:Arnold_Legendrians}] Observe that the Hausdorff convergence of $L_i$'s to $L$ implies that the set $ L_i \cap (O_N \times \R) $ is contained in an arbitrarily small neighbourhood of $ L \cap (O_N \times \R) $ for large $ i $. Because $\ell(a, L_i)$ corresponds to an intersection point of $L_i$ with the zero wall, we conclude that the set of limit points of $\{\ell(a, L_i): a \in H_*(N) \setminus \{0\}, i\in \N \}$ is contained in $ \spec(L)$.

Assume that $ \spec(L)$ has less than $ \cl(N) $ points. It follows from the above discussion that there exist $\alpha, \beta \in H_*(N) \setminus \{0\}$ and $ \lambda \in  \spec(L)$ such that for a subsequence $ (i_k) $ of indices we have $ \ell(\alpha, L_{i_k}) \rightarrow \lambda$ and $ \ell(\alpha \cap \beta, L_{i_k}) \rightarrow \lambda $ as $k \to \infty$. By passing to this subsequence, we may further assume that $ \ell(\alpha, L_{i}) \rightarrow \lambda$ and $ \ell(\alpha \cap \beta, L_{i}) \rightarrow \lambda $ as $i \to \infty$. Let us show that $ L \cap (O_N \times \{\lambda\}) $ is homologically non-trivial in $O_N \times \{\lambda\}$.

Pick any neighbourhood $ V $ of $ L \cap (O_N \times \{ \lambda \}) $ in $ J^1 N $.
Denote $ U := \pi_N (V) $, where $ \pi_N : J^1 N \rightarrow N $ is the natural projection, and pick a smooth function $f: N \rightarrow \R$ such that $f = 0$ on $ \overline{U}$ and $f < 0 $ on $N \setminus  \overline{U}$.

\begin{claim}\label{cl:claim2}
There exists an integer $i_0$ such that for any $i\geq i_0$, and for sufficiently small values of $\eps > 0$, $$\ell(\alpha \cap \beta, L_i +  L_{\eps f}) = \ell(\alpha \cap \beta, L_i).$$ 
\end{claim}
\begin{proof} By the Hausdorff convergence of $ L_i $ to $ L $, there exists some $ \delta > 0 $ such that for $i$ large enough and $\eps \geqslant 0 $ small enough, we have $$ (L_i+L_{\eps f}) \cap (O_N \times (\lambda - \delta,\lambda+\delta)) \subset V .$$ Furthermore, for any $(q,p,z)\in V$, we have that $q\in U$ and thus $f(q)=0$ and $df(q)=0$. This implies that $(L_i+L_{\eps f})\cap (O_N\times (\lambda - \delta,\lambda+\delta))=L_i\cap (O_N\times(\lambda - \delta,\lambda+\delta))$, in particular $ \spec(L_i + L_{\eps f}) \cap (\lambda - \delta,\lambda+\delta) = 
\spec(L_i) \cap (\lambda - \delta,\lambda+\delta) $.

The continuity and spectrality properties of spectral invariants, together with the fact that the spectrum of $L_i$ has an empty interior in $ \mathbb{R} $ and that $ \ell(\alpha\cap\beta,L_i) \in (\lambda-\delta,\lambda+\delta) $ for $ i $ large enough, imply that the spectral invariant $\ell(\alpha\cap\beta,L_i+L_{\eps f})$ is independent of $\eps$.
\end{proof}

Now the triangle inequality of Proposition \ref{prop:Leg_spec} implies that, for all $i$,  $\ell(\alpha \cap \beta, L_i +  L_{\eps f})-  \ell(\alpha , L_i) \leq \ell(\beta, L_{\eps f})$. Using the above claim, for $i$ large and $\eps$ small enough, we have $\ell(\alpha \cap \beta, L_i)-  \ell(\alpha , L_i) \leq \ell(\beta, L_{\eps f})$.
Taking limit as $i\to\infty$, and recalling that $\ell(\alpha \cap \beta, L_i),  \ell(\alpha,  L_i) \to \lambda$, we  obtain $0\leq \ell(\beta, L_{\eps f})$. 

We can now conclude our proof as follows. On the one hand, by Proposition \ref{prop:Leg_spec}.4,  we have $\ell(\beta,L_{\eps f})=c_{LS}(\beta, \eps f)$.  Note that $ c_{LS}(\beta, \eps f) =c_{LS}([N]\cap\beta, \eps f)$ and by the above paragraph this number is non-negative. On the other hand, Propostion \ref{prop:cLS}.2 gives $c_{LS}([N],\eps f)=0$. Thus, using Proposition \ref{prop:cLS}.3, we obtain  the equality $c_{LS}([N]\cap\beta, \eps f)=c_{LS}([N],\eps f)$. By Proposition \ref{prop:cLS}.4 it follows that the zero level set of $f$, that is the closure of $U= \pi_N(V)$, is homologically non-trivial in $ N $. Since our choice of a neighbourhood $ V $ of $ L \cap (O_N \times \{\lambda\}) $ was arbitrary, we conclude that $ L \cap (O_N \times \{\lambda\}) $ is homologically non-trivial in $O_N\times\{\lambda\}$.
\end{proof}

%
%

\bibliographystyle{abbrv}
\bibliography{c0LSLag}

\def\cprime{$'$}
\begin{thebibliography}{10}

\bibitem{buhovsky14}
L.~Buhovsky.
\newblock Towards the {$C^0$} flux conjecture.
\newblock {\em Algebr. Geom. Topol.}, 14(6):3493--3508, 2014.

\bibitem{BHS2}
L.~Buhovsky, V.~Humili\`ere, and S.~Seyfaddini.
\newblock {The action spectrum and $C^0$ symplectic topology}.
\newblock {\em arXiv:1808.09790}.

\bibitem{BHS}
L.~Buhovsky, V.~Humili\`ere, and S.~Seyfaddini.
\newblock A {$C^0$} counterexample to the {A}rnold conjecture.
\newblock {\em Invent. Math.}, 213(2):759--809, 2018.

\bibitem{chaperon}
M.~Chaperon.
\newblock On generating families.
\newblock In {\em The {F}loer memorial volume}, volume 133 of {\em Progr.
  Math.}, pages 283--296. Birkh\"auser, Basel, 1995.

\bibitem{chekanov}
Y.~V. Chekanov.
\newblock Critical points of quasifunctions, and generating families of
  {L}egendrian manifolds.
\newblock {\em Funktsional. Anal. i Prilozhen.}, 30(2):56--69, 96, 1996.

\bibitem{cornea-lupton-oprea}
O.~Cornea, G.~Lupton, J.~Oprea, and D.~Tanr{\'e}.
\newblock {\em Lusternik-{S}chnirelmann category}, volume 103 of {\em
  Mathematical Surveys and Monographs}.
\newblock American Mathematical Society, Providence, RI, 2003.

\bibitem{floer86}
A.~Floer.
\newblock Proof of the {A}rnol\cprime d conjecture for surfaces and
  generalizations to certain {K}\"ahler manifolds.
\newblock {\em Duke Math. J.}, 53(1):1--32, 1986.

\bibitem{floer88}
A.~Floer.
\newblock The unregularized gradient flow of the symplectic action.
\newblock {\em Comm. Pure Appl. Math.}, 41:775--813, 1988.

\bibitem{floer89b}
A.~{Floer}.
\newblock {Cuplength estimates on Lagrangian intersections.}
\newblock {\em {Commun. Pure Appl. Math.}}, 42(4):335--356, 1989.

\bibitem{floer89}
A.~Floer.
\newblock Symplectic fixed points and holomorphic spheres.
\newblock {\em Comm. Math. Phys.}, 120(4):575--611, 1989.

\bibitem{FW}
B.~Fortune and A.~Weinstein.
\newblock A symplectic fixed point theorem for complex projective spaces.
\newblock {\em Bull. Amer. Math. Soc. (N.S.)}, 12(1):128--130, 1985.

\bibitem{hofer}
H.~Hofer.
\newblock Lusternik-{S}chnirelman-theory for {L}agrangian intersections.
\newblock {\em Ann. Inst. H. Poincar\'e Anal. Non Lin\'eaire}, 5(5):465--499,
  1988.

\bibitem{hofer90}
H.~Hofer.
\newblock On the topological properties of symplectic maps.
\newblock {\em Proc. Roy. Soc. Edinburgh Sect. A}, 115(1-2):25--38, 1990.

\bibitem{Homma}
T.~Homma.
\newblock An extension of the {J}ordan curve theorem.
\newblock {\em Yokohama Math. J.}, 1:125--129, 1953.

\bibitem{Howard2012}
W.~{Howard}.
\newblock {Action Selectors and the Fixed Point Set of a Hamiltonian
  Diffeomorphism}.
\newblock {\em ArXiv: 1211.0580}, Nov. 2012.

\bibitem{HLS13}
V.~Humili{\`e}re, R.~Leclercq, and S.~Seyfaddini.
\newblock Coisotropic rigidity and {$C^0$}-symplectic geometry.
\newblock {\em Duke Math. J.}, 164(4):767--799, 2015.

\bibitem{Humiliere-Vichery}
V.~Humili{\`e}re and N.~Vichery.
\newblock Cuplength estimates via microlocal sheaf theory. {I}n preparation.

\bibitem{kawamoto}
Y.~{Kawamoto}.
\newblock {On $C^0$-continuity of the spectral norm on non-symplectically
  aspherical manifolds}.
\newblock {\em arXiv:1905.07809}, May 2019.

\bibitem{lalonde-mcduff}
F.~Lalonde and D.~McDuff.
\newblock The geometry of symplectic energy.
\newblock {\em Ann. of Math. (2)}, 141(2):349--371, 1995.

\bibitem{LMP}
F.~Lalonde, D.~McDuff, and L.~Polterovich.
\newblock On the flux conjectures.
\newblock In {\em Geometry, topology, and dynamics ({M}ontreal, {PQ}, 1995)},
  volume~15 of {\em CRM Proc. Lecture Notes}, pages 69--85. Amer. Math. Soc.,
  Providence, RI, 1998.

\bibitem{laudenbach-sikorav}
F.~Laudenbach and J.-C. Sikorav.
\newblock Persistance d'intersection avec la section nulle au cours d'une
  isotopie hamiltonienne dans un fibr\'e cotangent.
\newblock {\em Invent. Math.}, 82(2):349--357, 1985.

\bibitem{Le-Ono}
H.~V. L\^{e} and K.~Ono.
\newblock Cup-length estimates for symplectic fixed points.
\newblock In {\em Contact and symplectic geometry ({C}ambridge, 1994)},
  volume~8 of {\em Publ. Newton Inst.}, pages 268--295. Cambridge Univ. Press,
  Cambridge, 1996.

\bibitem{Lec}
R.~Leclercq.
\newblock Spectral invariants in {L}agrangian {F}loer theory.
\newblock {\em J. Mod. Dyn.}, 2(2):249--286, 2008.

\bibitem{Lec-Zap}
R.~Leclercq and F.~Zapolsky.
\newblock Spectral invariants for monotone {L}agrangians.
\newblock {\em J. Topol. Anal.}, 10(3):627--700, 2018.

\bibitem{LS}
L.~{Lusternik} and L.~{Schnirelmann}.
\newblock {M\'ethodes topologiques dans les probl\`emes variationnels.}
\newblock {Moskau: Issledowatelskij Institut Mathematiki i Mechaniki pri J. M.
  G. U (1930).}, 1930.

\bibitem{matsumoto}
S.~Matsumoto.
\newblock Arnold conjecture for surface homeomorphisms.
\newblock In {\em Proceedings of the {F}rench-{J}apanese {C}onference
  ``{H}yperspace {T}opologies and {A}pplications'' ({L}a {B}ussi\`ere, 1997)},
  volume 104, pages 191--214, 2000.

\bibitem{MVZ}
A.~Monzner, N.~Vichery, and F.~Zapolsky.
\newblock Partial quasimorphisms and quasistates on cotangent bundles, and
  symplectic homogenization.
\newblock {\em J. Mod. Dyn.}, 6(2):205--249, 2012.

\bibitem{Oh99}
Y.-G. Oh.
\newblock Symplectic topology as the geometry of action functional. {II}.
  {P}ants product and cohomological invariants.
\newblock {\em Comm. Anal. Geom.}, 7(1):1--54, 1999.

\bibitem{Oh05}
Y.-G. Oh.
\newblock Construction of spectral invariants of hamiltonian paths on closed
  symplectic manifolds.
\newblock {\em The breadth of symplectic and Poisson geometry. Progr. Math.
  \textbf{232}, Birkhauser, Boston}, pages 525--570, 2005.

\bibitem{muller-oh}
Y.-G. Oh and S.~M{\"u}ller.
\newblock The group of {H}amiltonian homeomorphisms and {$C^0$}--symplectic
  topology.
\newblock {\em J. Symplectic Geom.}, 5(2):167--219, 2007.

\bibitem{Pageault}
P.~Pageault.
\newblock Functions whose set of critical points is an arc.
\newblock {\em Math. Z.}, 275(3-4):1121--1134, 2013.

\bibitem{polterovich93}
L.~Polterovich.
\newblock {Symplectic displacement energy for Lagrangian submanifolds}.
\newblock {\em Ergodic Theory and Dynamical Systems}, 13:357--367, 1993.

\bibitem{rudyak-oprea}
Y.~B. Rudyak and J.~Oprea.
\newblock On the {L}usternik-{S}chnirelmann category of symplectic manifolds
  and the {A}rnold conjecture.
\newblock {\em Math. Z.}, 230(4):673--678, 1999.

\bibitem{sandon12}
S.~Sandon.
\newblock On iterated translated points for contactomorphisms of {$\Bbb
  R^{2n+1}$} and {$\Bbb R^{2n}\times S^1$}.
\newblock {\em Internat. J. Math.}, 23(2):1250042, 14, 2012.

\bibitem{sandon13}
S.~Sandon.
\newblock A {M}orse estimate for translated points of contactomorphisms of
  spheres and projective spaces.
\newblock {\em Geom. Dedicata}, 165:95--110, 2013.

\bibitem{schwarz}
M.~Schwarz.
\newblock On the action spectrum for closed symplectically aspherical
  manifolds.
\newblock {\em Pacific J. Math.}, 193(2):419--461, 2000.

\bibitem{Sey12}
S.~Seyfaddini.
\newblock Descent and {$C^0$}-rigidity of spectral invariants on monotone
  symplectic manifolds.
\newblock {\em J. Topol. Anal.}, 4(4):481--498, 2012.

\bibitem{Sey13c}
S.~Seyfaddini.
\newblock A note on {$C^0$} rigidity of {H}amiltonian isotopies.
\newblock {\em J. Symplectic Geom.}, 11(3):489--496, 2013.

\bibitem{shelukhin-zoll}
E.~{Shelukhin}.
\newblock {Viterbo conjecture for Zoll symmetric spaces}.
\newblock {\em arXiv:1811.05552}, Nov 2018.

\bibitem{shelujhin-string}
E.~{Shelukhin}.
\newblock {String topology and a conjecture of Viterbo}.
\newblock {\em arXiv:1904.06798}, Apr 2019.

\bibitem{theret}
D.~Th\'eret.
\newblock A complete proof of {V}iterbo's uniqueness theorem on generating
  functions.
\newblock {\em Topology Appl.}, 96(3):249--266, 1999.

\bibitem{viterbo}
C.~Viterbo.
\newblock Symplectic topology as the geometry of generating functions.
\newblock {\em Math. Annalen}, 292:685--710, 1992.

\bibitem{Zapolsky}
F.~Zapolsky.
\newblock Geometry of contactomorphism groups, contact rigidity, and contact
  dynamics in jet spaces.
\newblock {\em Int. Math. Res. Not. IMRN}, (20):4687--4711, 2013.

\end{thebibliography}


{\small

\medskip
\noindent Lev Buhovski\\
School of Mathematical Sciences, Tel Aviv University \\
{\it e-mail}: levbuh@tauex.tau.ac.il
\medskip

\medskip
\noindent Vincent Humili\`ere \\
\noindent CMLS, Ecole Polytechnique, Institut Polytechnique de Paris, 91128 Palaiseau Cedex, France\\
{\it e-mail:} vincent.humiliere@polytechnique.edu
\medskip

\medskip
 \noindent Sobhan Seyfaddini\\
\noindent Sorbonne Universit\'e, Universit\'e de Paris, CNRS, Institut de Math\'ematiques de Jussieu-Paris Rive Gauche, F-75005 Paris, France.\\
{\it e-mail:} sobhan.seyfaddini@imj-prg.fr}

\end{document}